\documentclass[12pt]{article}
\usepackage{amsmath, amsfonts, amsthm, amssymb, verbatim}
\usepackage{graphicx}
\usepackage[mathcal]{euscript}
\usepackage{amstext}
\usepackage{amssymb,mathrsfs}
\usepackage[latin1]{inputenc}
\newcommand{\R}{\mathbb R}

  \newcommand{\E}{\mathbb E}

\newtheorem{theorem}{Theorem}[section]

 \newtheorem{remark}[theorem]{Remark}

\newtheorem{corollary}[theorem]{Corollary}
\newtheorem{definition}[theorem]{Definition}

\begin{document}
\title{  Probabilistic  representation for mild solution of the Navier-Stokes equations }

\author{ C. Olivera\thanks{Departamento de Matem\'{a}tica, Universidade Estadual de
Campinas, Brazil. E-mail: \textsl{colivera@ime.unicamp.br}. }}

\date{}

\maketitle

\textit{Key words and phrases. Navier-Stokes equations,  
Stochastic differential equation, Iyer-Cosntantin representation formula,  Mild Solution.}


%
\begin{abstract}
This paper is based on a formulation of the Navier-Stokes equations developed by Iyer and  Constantin \cite{Cont} , where the velocity field of a viscous incompressible 
fluid is written as the expected value of a stochastic process. Our contribution  is to establish this  probabilistic representation formula for mild solutions of  the Navier-Stokes equations on 
$\R^{d} $.  

 \end{abstract}
%
\maketitle

\section {Introduction} \label{Intro}

We consider the classical Cauchy problem for the Navier-Stokes system,
describing the evolution of a velocity field
$u$ of an incompressible fluid with kinematic
viscosity $\nu$

\begin{equation}\label{Navier}
 \left \{
\begin{aligned}
    &\partial_t u(t, x) =\nu  \Delta u(t, x) - [u(t, x)\cdot \nabla] u(t, x)  - \nabla \pi(t, x)
    \\& \text{ div } u(t, x)=0
    \\    &u(0,x)=  u_{0}(x)
\end{aligned}
\right .
\end{equation}

The unknown quantities are the velocity $u(t, x) = (u_{1}(t, x), . . . , u_{d}(t, x))$ of
the fluid element at time $t$ and position $x$ and the pressure $\pi(t, x)$. Such equations always attract the attention of
many researchers, with an enormous quantity of publications in the literature. Concerning
classical results about (1.1), we refer to the book by Teman \cite{Tema}.
In the 1960s, mild solutions were first constructed by Kato and Fujita (\cite{Kato3} and 
\cite{Kato2}) that are continuous in time and take values in the Sobolev spaces
$u\in C([0,T], H^{s}(\R^{d}))$ ($s\geq  \frac{d}{2}-1$). Results on the
existence of mild solutions with value in $L^{p}$, were established 
by Fabes, Jones and Riviere \cite{Fabes} and by  Giga \cite{Giga2}.
 In 1992, a modern treatment for mild solutions in $H^{s}$  was given by Chemin \cite{Che}. 
For recent developments see   Lemarie-Rieusset \cite{Lema}. 
One of the (still open) million dollar problems posed by the Clay Institute  is to show that given a smooth initial data $u_0$ the
solution to  (\ref{Navier}) in three dimensions remains smooth for all time. We are interested
in developing probabilistic techniques, that could help solve this problem

\bigskip

Probabilistic representations of solutions of partial differential equations as the
expected value of functionals of stochastic processes date back to the work of Einstein,
Feynman, Kac, and Kolmogorov in physics and mathematics. The Feynman-Kac
formula is the most well-known example, which has provided a link between
linear parabolic partial differential equations and probability theory, see \cite{Kara}. 
In 2008 Constantin and Iyer \cite{Cont}(see also \cite{Cont2} and \cite{Iyer}) established
a probabilistic Lagrangian representation formula by making use of stochastic flows. They show that 
$u$ is  classical solution to the Navier-Stokes equation (\ref{Navier}) if an only if $u$ satisfies the stochastic systems

\begin{equation}\label{itoass}
X_{t}(x)= x + \int_{0}^{t}   u(r, X_{r}(x)) \ dr  +  B_{t},
\end{equation}

\begin{equation}\label{repre}
u(t,x)= \mathbb{P}  \E[   (\nabla X_{t}^{-1})^{\ast}  u_{0}(X_{t}^{-1}) ]
\end{equation}

where $B_{t}$ denoting the standard Brownian motion,  $\mathbb{P}$ is the Leray-Hodge projection and $\ast$ denotes the transposition of matrix.  We mention that Fang, D Luo \cite{Fang}  obtained formula  (\ref{repre})  on  a compact manifolds,  Rezakhanlou \cite{Reza} wrote
 the representation (\ref{repre})   in the context of symplectic  geometry and Zhang \cite{Zhang3}
extended  the formula (\ref{repre}) for non-local operators.  Different probabilistic representations of the solution of the the Navier-Stokes equations were studied by S. Albeverio, Y. Belopolskaya \cite{Albe}, Busnello \cite{Bu},  Busnello, Flandoli, Romito \cite{Bu2}  Cipriano, Cruzeiro \cite{Cip}, Cruzeiro, Shamarova \cite{Cru} and  Zhang  \cite{Zhang3}. 

\bigskip

 Strong solutions to  the equation (\ref{itoass}) are known for irregular $u$
, the best result (after previous investigations of Zvonkin \cite{Zo}, Veretennikov\cite{Ve}, among others) being proved by
Krylov, R\"{o}ckner in \cite{Krylov}.  More recently Flandoli, Gubinelli, Priola,  see \cite{FGP2} and \cite{FGP}, 
proved that   if the drift term is  H\"{o}lder  continuous then 
  $x \rightarrow X_{s,t} $ is a $C^{1}$- stochastic flow. 

\bigskip The contribution of this paper is to show that 
the unique mild solution of the equation (\ref{Navier})  with values in $ C([0,T], H^{s}(\R^{d}))$ has the stochastic representation (\ref{repre}). The proof is simple and it is  based in stability properties of the mild solution 
and in the  flow properties  associated to the equation (\ref{itoass}). The result is the following theorem.

\bigskip
In fact, through of this paper, we fix a stochastic basis with a
$d$-dimensional Brownian motion $\big( \Omega, \mathcal{F}, \{
\mathcal{F}_t: t \in [0,T] \}, \mathbb{P}, (B_{t}) \big)$.
We denoted $M$ a  generic constant.

\section{Preliminaries} \label{Intro}

\subsection{Mild Solution.}

In this subsection we  recall some results on the Stokes operator. 

\[
L_{\sigma}^{2}= \  the  \ closure  \ in \ [L^{2}(\R^{d})]^d  \ of  \{ u\in [C_{0}^{\infty}(\R^{d})]^d, \ div u=0\}
\]
and
\[
G^{2}= \{ \nabla q,   q\in W^{1,2}(\R) \}.
\]

\bigskip

We then have the following Helmholtz decomposition 
\[
[L^{2}(\R)]^d= L_{\sigma}^{2} \oplus G^{2}, 
\]
 the sum above reduces to the orthogonal decomposition and $L_{\sigma}^2$
is a separable Hilbert space, whose scalar product is denoted by $(\cdot,\cdot)$.

\bigskip
Let $\mathbb{P}$ be the continuous projection  from $L^{2}(\R^{d})$ to  $L_{\sigma}^{2}$
associated with this decomposition  and let $\Delta$ be
the Laplace operator. Now, we define the Stokes operator $A$ in $L_{\sigma}^{2}$ by $A=-\mathbb{P}\Delta$.  The operator $-A$ generates  a bounded analytic semigroup  $\{S(t)\}_{t\ge 0}$, see \cite{Soh}.

\bigskip
The potential space $H^{s}(\R^{d})$ is defined as the space 
$(\mathit{I}-\Delta)^{-s/2}L^{2}$ equipped with the norm 
$\| f\|_{H^{s}}:=\| (\mathit{I}-\Delta)^{s/2}\|_{L^{2}}$. It is well know that

\[
\| fg\|_{H^{s}}\leq  M_{s} \  \| f\|_{H^{s}} \ \| g\|_{H^{s}} \ \ \ \  if \  s>\frac{d}{2}
\]

We also  recall that

\begin{equation}\label{stimaSpr}
\| S(t) u \|_{H^{s}}  \le  \  M_{s} \| u \|_{H^{s}} \  
\end{equation}

\begin{equation}\label{stimaASp}
\|A^{\alpha} S(t) u \|_{H^{s}}\le \  \frac{M_{s}}{t^{\alpha}} \ \| u \|_{H^{s}} .
\end{equation}

 We consider  the Navier -Stokes initial problem  in the space
$H^{s}(\R^{d})$. Applying the projection operator $\mathbb{P}$ to \eqref{Navier}
we get rid of the pressure term; setting $\nu=1$, equation \eqref{Navier} becomes
\begin{equation}\label{eq-abst}
\begin{cases}
du(t) + Au(t)\ dt= B(u(t)) \ dt, & t>0 
\\ u(0)=u_{0}
\end{cases}
\end{equation}
where the non linear term $B$ is defined by $B(u)=-\mathbb{P}[(u\cdot \nabla)u]$. 
Since $u$ is divergence free, we also have the representation
$B(u)= -\mathbb{P}[ \text{div}\ (u\otimes u )]$ which will be useful later on.
We consider the mild solution

\begin{definition}\label{defimild}  A measurable function $u: [0,T]\rightarrow  H^{s}(\R^{d})$
is a mild solution of the equation (\ref{eq-abst}) if

\begin{enumerate}
\item $u \in C(0,T; H^{s}(\R^{d}))$, 

\item  for all $t \in (0,T]$, we have
\begin{equation}\label{mild}
u(t)= S(t)u_{0}+  \ \int_{0}^{t} S(t-s)B(u(s)) \  ds 
 \end{equation}
\end{enumerate}

\end{definition}

\begin{definition} We assume that there exist $T>0$ such that satisfies $u$ satisfies items 1 and 2  in  definition  \ref{defimild}. Then 
we called $u$ of local mild solution. 
\end{definition}

\subsection{Stochastic Flows.}

In this subsection we follow the seminar paper by Flandoli, Gubinelli and Priola in \cite{FGP2}.  We consider the SDE

\begin{equation}
\label{sde3}
dX_{t} (x)= b(t,X_{s,t}(x) ) dt +  d B_{t}, \ X_{s}=x\in\R^{d},
\end{equation}

where $X_{s,t}(x)= X(s,t,x)$, also $X_{t}(x)= X(0,t,x)$. Moreover, the inverse $Y_{s,t}(x):=X_{s,t}^{-1}(x)$ satisfies the
following backward stochastic differential equation%

\begin{equation}
\label{itoassBac}Y_{s,t}(x)= x - \int_{s}^{t} b(r,  Y_{r,t}(x)) \ dr - (B_{t}-B_{s}). 
\end{equation}

We denote  by $\phi_{s,t}$ the flow associated to $X_{s,t}$ and $\psi_{s,t}$ its inverse.
Let $T> 0 $ be be fixed. For any $\alpha\in (0,1)$,  we denoted by $L^{\infty}([0,T],C_{b}^{\alpha}(\R^{d}))$ the space bounded  Borel functions  $f : [0,T]\times \R^{d}\rightarrow \R$ such that

\begin{align*}
\|f\|_{\alpha,T}:=  \sup_{t\in [0,T]}  \sup_{x\neq y, |x-y|\leq 1} \frac{|f(t,x)-f(t,y)|}{|x-y|^\theta} < \infty\,.
\end{align*}

We  also recall   the important results  in \cite{FGP2}.

\begin{theorem}\label{difH} We assume that $b \in   L^{\infty}([0,\infty),C_{b}^{\alpha}(\R^{d}))$. Then 

\begin{itemize}
\item[a)] There exists a unique solution of the SDE (\ref{sde3}). 

\item[b)]
 There exists a stochastic flow $\phi_{s,t}$ of diffeomorphisms associated to  equation
(\ref{sde3}). The flow is the class $C^{1+\alpha^{\prime}}$ for every  $0<\alpha^{\prime}< \alpha$.

\item[c)] Let $b^{n}\in L^{\infty}([0,\infty),C_{b}^{\alpha}(\R^{d}))$ be a sequence of the vector field and 
$\phi^{n}$ be the corresponding stochastic flow. If $b^{n}\rightarrow b$ in $L^{\infty}([0,\infty),C_{b}^{\alpha}(\R^{d}))$, then 
for any $p\geq 1$ we have 

\begin{equation}\label{es1}
\lim_{n\rightarrow \infty} \sup_{x\in \R^{d}} \sup_{s\in[0,T]} \E [\sup_{r\in[s,T]}|\phi^{n}_{s,r}-\phi_{s,r}|^{p}]=0,
\end{equation}

\begin{equation}\label{es2}
 \sup_{n}\sup_{x\in \R^{d}} \sup_{s\in[0,T]} \E [\sup_{r\in[s,T]}|D\phi_{s,r}^{n}|^{p}]< \infty,
\end{equation}

\begin{equation}\label{es3}
\lim_{n\rightarrow \infty} \sup_{x\in \R^{d}} \sup_{s\in[0,T]} \E [\sup_{r\in[s,T]}|D\phi_{s,r}^{n}-D\phi_{s,t}|^{p}]=0.
\end{equation}
\end{itemize} 
\end{theorem}

\begin{remark}
The same results are valid for the backward flows $\psi_{s,t}^{n}$ and  $\psi_{s,t}$ since 
are solutions of the same SDE driven by the drifts $-b_{n}$ and $-b$.
\end{remark}

\section{Result}

 Let $\{\rho_n\}_n$ be a family of standard symmetric mollifiers.
 We define the family of regularized initial data  as $u_{0}^{n}(x) = (u \ast \rho_\varepsilon) (x) $. Let $T>0$. Now, we  assume that 
for all $n$ there exist   $u^{n}$ a  classical  solution in  $[0,T]\times \R^{d}$ of 
 
\begin{equation}\label{eq-abstreg}
\begin{cases}
du^{n}(t) + Au^{n}(t)\ dt= B(u^{n}(t)) \ dt, & t>0 
\\ u(0)=u_{0}^{n}. 
\end{cases}
\end{equation}

\begin{theorem} \label{T1}  We fix $T>0$ and  we assume $u_{0}\in H^{s}(\R^{d})$. Let be $u \in  C([0,T], H^{s}(\R^{d}))$ a local  mild solution  with  $s> \frac{d}{2}$ such that $u^{n}$ converge to $u$ in $C([0,T], H^{s}(\R^{d}))$. Then we have 
 that 

\begin{equation}\label{repr2}
u(t,x)= \mathbb{P}  \E[   (\nabla X_{t}^{-1})^{\ast}  u_{0}(X_{t}^{-1}) ]. 
\end{equation}
\end{theorem}

\begin{proof} {\it Step 1 : Regular initial data.}

  By It\^o formula or by  Constantin-Iyer \cite{Cont2} we have

\[
u^{n}(t,x)= \mathbb{P}  \E[   (\nabla Y_{t}^{n})^{\ast}  u_{0}^{n}(Y_{t}^{n}) ]. 
\]

where $Y_{t}^{n}$ is the inverse of 

\[
dX_{t}^{n}= u^{n}(t,X_{t}^{n}) dt +  d B_{t}, \ X_{0}=x\in\R^{d}.
\]

\bigskip

{\it Step 2 : Convergence II.}
\bigskip

 From  $H^{s}(\R^{d})\subset C_{b}^{\alpha}(\R^{d})$ with $\alpha=s-\frac{d}{2}$, 
 hypothesis  and theorem \ref{difH} we have 

\begin{equation}\label{c1}
\lim_{n\rightarrow \infty} \sup_{x\in \R^{d}} \E [\sup_{t\in[0,T]}|Y^{n}_{t}-Y_{t}|^{p}]=0,
\end{equation}

\begin{equation}\label{c3}
\lim_{n\rightarrow \infty} \sup_{x\in \R^{d}}  \E [\sup_{t\in[0,T]}|DY_{t}^{n}-DY_{t}|^{p}]=0,
\end{equation}

\begin{equation}\label{c4}
 \sup_{x\in \R^{d}}  \E [\sup_{t\in[0,T]}|DY_{t}^{n}|^{p}]< \infty,
\end{equation}

where $Y_{t}$ is the inverse of $X_{t}$ and  it  verifies (\ref{sde3})  with drift $u(t,x)$.

\bigskip

 {\it Step 3 : Convergence III.}
\bigskip

We observe  that

\begin{align*}
&  \E[   (\nabla Y_{t}^{n})^{\ast}  u_{0}^{n}(Y_{t}^{n}) ]- \E[   (\nabla Y_{t})^{\ast}  u_{0}(Y_{t}) ]| 
\\ & 
\leq\E[   (\nabla Y_{t}^{n})^{\ast}  u_{0}^{n}(Y_{t}^{n})-(\nabla Y_{t}^{n})^{\ast}  u_{0}(Y_{t}^{n})] |  
\\ & 
+  \ \E[   (\nabla Y_{t}^{n})^{\ast}  u_{0}(Y_{t}^{n})-(\nabla Y_{t}^{n})^{\ast}  u_{0}(Y_{t})] | 
\\ & 
+ \ \E[   (\nabla Y_{t}^{n})^{\ast}  u_{0}(Y_{t})-(\nabla Y_{t})^{\ast}  u_{0}(Y_{t})] | \   
\\ & = I_{1}+ I_{2} + I_{3}. 
\end{align*}

\bigskip

  By H\"{o}lder inequality and  (\ref{c4}) we have

\begin{align*}
&  \int_{\R^{d}} |I_{1}|^{2} \ dx \\ &
 \leq 
\int_{\R^{d}} \E  |\nabla Y_{t}^{n}|^{2}  \   \E |u_{0}^{n}(Y_{t}^{n})- u_{0}(Y_{t}^{n})|^{2}   \ dx  
\\ & \leq \sup_{x,t} \E  |\nabla Y_{t}^{n}|^{2}  \  \int_{\R^{d}} \E |u_{0}^{n}(Y_{t}^{n})- u_{0}(Y_{t}^{n})|^{2}   \ dx  
\\ &  = C \  \int_{\R^{d}} \E |u_{0}^{n}(x)- u_{0}(x)|^{2}   \ dx  
\end{align*}

it follows that $I_{1} \rightarrow 0 $ in $C([0,T], L^2({\R^{d}}))$.

\bigskip

By  H\"{o}lder inequality and   (\ref{c4})  we obtain

\begin{align*}
& \int_{\R^{d}}  |I_{2}|^{2} \ dx  \\ & \leq  \ \int_{\R^{d}} \E   |\nabla Y_{t}^{n}|^{2} \   \E |u_{0}(Y_{t}^{n})-  u_{0}(Y_{t}) |^{2} \   dx   \\ & \leq \sup_{x,t} \E |\nabla Y_{t}^{n}|^{2} \  \int_{\R^{d}}\E |u_{0}(Y_{t}^{n})-  u_{0}(Y_{t}) |^{2} \   dx, 
\end{align*}

  from   (\ref{c1}) and dominated convergence we get that $I_{2}\rightarrow 0 $ in $C([0,T], L^2({\R^{d}}))$.

\bigskip

We observe that

\begin{align*}
&  \int_{\R^{d}} |I_{3}|^{2} \ dx \\ &
\leq \  \int_{\R^{d}} \E| (\nabla Y_{t}^{n})^{\ast}- (\nabla Y_{t})^{\ast}|^{2} \E |u_{0}(Y_{t}) |^{2} \   dx 
\\ & \leq C  \sup_{x,t}\E| (\nabla Y_{t}^{n})^{\ast}- (\nabla Y_{t})^{\ast}|^{2} \ dx 
\end{align*}

from (\ref{c3})  we deduce that  $I_{3}\rightarrow 0 $ in $C([0,T], L^2({\R^{d}}))$.

\bigskip

Thus  we conclude that  $ \E[ (\nabla Y_{t}^{n})^{\ast} u_{0}^{n}(Y_{t}^{n})]\rightarrow  \E [(\nabla Y_{t})^{\ast}  u_{0}(Y_{t})]$ strong in $C([0,T], L^2({\R^{d}}))$. This  implies that
 $u^{n}= \mathbb{P}  \E [(\nabla Y_{t}^{n})^{\ast}  u_{0}^{n}(Y_{t}^{n})]$converge to $ \mathbb{P}  \E [(\nabla Y_{t})^{\ast}  u_{0}(Y_{t})]$  in $C([0,T], L^{2}(\R^{d}))$.

\bigskip

{\it Step 4 : Conclusion.}

\bigskip
From  step I and hypothesis  we conclude that $u(t,x)=  \mathbb{P} \E [(\nabla Y_{t})^{\ast} u_{0}(Y_{t})]$.

\end{proof}

 We observed that by Kato construction of the mild solution we can take $T$ sufficiently small such that 
there exists an unique mild solution  in $C(0,T, ;H^{s}(\R^{d}))$ with initial conditions 
$u_{0}$ and $u_{0}^{n}$, see for instance  \cite{Lema}.

\begin{corollary} We assume $u_{0}\in H^{s}(\R^{d})$ and $T$ is small enough. Let be $u \in  C([0,T], H^{s}(\R^{d}))$ the unique local mild solution 
with  $s> \frac{d}{2}$. Then we have 
 that 

\begin{equation}\label{repr3}
u(t,x)= \mathbb{P}  \E[   (\nabla X_{t}^{-1})^{\ast}  u_{0}(X_{t}^{-1}) ]. 
\end{equation}
\end{corollary}

\begin{proof}

\bigskip
  Now, we consider $u^{n}$ the unique local  mild solution of 
 
\begin{equation}\label{eq-abstreg2}
\begin{cases}
du^{n}(t) + Au^{n}(t)\ dt= B(u^{n}(t)) \ dt, & t>0 
\\ u(0)=u_{0}^{n}. 
\end{cases}
\end{equation}

in $[0,T]$.  We have

\[
 u(t)- u^{n}(t)=u_{0}- u_{0}^{n} +  \ \int_{0}^{t} S(t-s) \big(B(u(s))-B( u^{n}(s)) \big)  \ ds .  
\]
 
By classical estimations we obtain

\begin{align*}
& \| u(t)-\ u^{n}(t)\|_{H_{s}} \\ &  \leq \|u_{0}- u_{0}^{n}\|_{H_{s}}  \\ &
+   \int_{0}^{t} \|S(t-s) \big(B(u(s))-B( u^{n}(s)) \big)\|_{H_{s}}  \ ds 
\\ & \le \|u_{0}- u_{0}^{n}\|_{H_{s}}  \\ &  + \int_{0}^{t} \frac{M}{(t-s)^{\frac{1}{2}}} (\| u^{n}(s)\|_{H_{s}}+ \| u(s)\|_{H_{s}})  \|u(s)-u^{n}(s)\|_{H_{s}}  \ ds.   
\end{align*}

It is well know that

\[
\| u(t)\|_{H^{s}}\leq M \|u_0\|_{H^{s}}
\]

and 

\[
\| u^{n}(t)\|_{H^{s}}\leq M \|u_0^{n}\|_{H^{s}}\leq  M\|u_0\|_{H^{s}}. 
\]

Thus we have 

\begin{align*}
& \sup_{t\in [0,T]}\| u(t)-\ u^{n}(t)\|_{H_{s}}\\ &
 \leq \|u_{0}- u_{0}^{n}\|_{H_{s}}  M  T^{\frac{1}{2}}  (\sup_{t\in [0,T]}\| u^{n}(t)\|_{H_{s}}+ 
 \sup_{t\in [0,T]}\| u(t)\|_{H_{s}})  \\ &  \times \sup_{t\in [0,T]} \|u(s)-u^{n}(s)\|_{H_{s}}
\\ & \leq \|u_{0}- u_{0}^{n}\|_{H_{s}}
 + M  T^{\frac{1}{2}} \| u_{0}\|_{H_{s}} \sup_{t\in [0,T]} \|u(s)-u^{n}(s)\|_{H_{s}}.
\end{align*}

If  $M  T^{\frac{1}{2}} \| u_{0}\|_{H_{s}} < 1$ we get

\[
\sup_{t\in [0,T]}\| u(t)-\ u^{n}(t)\|_{H_{s}}\leq C  \|u_{0}- u_{0}^{n}\|_{H_{s}}. 
\]

Then we  deduce 

\[
\sup_{t\in [0,T]}\| u(t)-\ u^{n}(t)\|_{H_{s}}\rightarrow 0 \ as  \ n\rightarrow\infty. 
\]

The representation (\ref{repr3})  we follow from the 
theorem \ref{T1}. 

\end{proof}

\section*{Acknowledgements}

    Christian Olivera  is partially supported by FAPESP 
		by the grants 2017/17670-0 and 2015/07278-0 . Also supported by CNPq by the grant
		426747/2018-6. 


\end{document}